\documentclass[12pt]{article}
\usepackage{mathrsfs}
\usepackage{amsmath}
\usepackage{amsthm}
\usepackage{bbm}
\usepackage{subfloat}
\usepackage{subfigure}
\usepackage{fancyhdr,graphicx}
\usepackage{amsfonts}
\usepackage{amssymb}
\usepackage{latexsym,bm}
\usepackage{newlfont}
\usepackage{float}
\usepackage{multirow}
\usepackage[numbers,sort&compress]{natbib}
\allowdisplaybreaks[3]

\newtheorem{thm}{Theorem}[section]
\newtheorem{lemma}[thm]{Lemma}
\newtheorem{cor}[thm]{Corollary}

\makeatletter
\long\def\@makecaption#1#2{%
 \vskip\abovecaptionskip
  \sbox\@tempboxa{{#1.}\quad #2}%
 \ifdim \wd\@tempboxa >\hsize
    { #1.}\quad #2\par
     \else
  \global \@minipagefalse
   \hb@xt@\hsize{\hfil\box\@tempboxa\hfil}%
   \fi
   \vskip\belowcaptionskip}
\makeatother

\title{{Fullerenes with the maximum Clar number\thanks{This
work is supported by NSFC (Grant Nos.11371180 and 11401279).}}}
\author{Yang Gao, Qiuli Li, Heping Zhang\footnote{Corresponding author.}
\\\small{School of Mathematics and Statistics, Lanzhou
University, Lanzhou, Gansu 730000, China}
\\\small{E-mail addresses:  gaoy12@lzu.edu.cn, qlli@lzu.edu.cn, zhanghp@lzu.edu.cn}}

\date{}

\begin{document}
\maketitle
\makeatletter
\newcommand{\rmnum}[1]{\romannumeral #1} ¡¡¡¡
\newcommand{\Rmnum}[1]{\expandafter\@slowromancap\romannumeral #1@}
\makeatother

\begin{abstract}
 The Clar number of a fullerene is the maximum number of independent resonant hexagons in the fullerene. It is known that the
Clar number
 of a
fullerene with $n$ vertices is bounded above by $\lfloor n/6\rfloor-2$.
 We find that there are no fullerenes with $n\equiv 2\pmod 6$ vertices
 attaining this bound. In other words, the Clar number for a fullerene with $n\equiv 2\pmod 6$
vertices is bounded above by $\lfloor n/6\rfloor-3$. Moreover, we show that two
experimentally produced fullerenes C$_{80}$:1 (D$_{5d}$) and
C$_{80}$:2 (D$_{2}$) attain this bound. Finally, we present a
graph-theoretical characterization for fullerenes, whose order $n$ is congruent
to 2 (respectively, 4) modulo 6, achieving the maximum Clar number
$\lfloor n/6\rfloor-3$ (respectively,
$\lfloor n/6\rfloor-2$).


\end{abstract}


\section{Introduction}

Clar number is a stability predictor of the benzenoid
hydrocarbon isomers. The concept of Clar number  originates from
the Clar's sextet theory \cite{Clar1972} and Randi\'{c}'s conjugated circuits model \cite{Randic2003}. For any two isomeric benzenoid hydrocarbons,
the one with larger Clar number
 is more stable \cite{Clar1972,King1993}. Hansen and Zheng \cite{Hansen1994} reduced the Clar number problem of benzenoid hydrocarbons to an integer linear programming. Based on abundant computation, the same  authors conjectured the linear programming relaxing is sufficient.
The conjecture was confirmed by  Abeledo and Atkinson \cite{Abeledo2007}.

A \textsf{fullerene} is defined as a finite,
trivalent plane graph consisting solely of pentagons and hexagons. Gr\"{u}nbaum  and Motzkin \cite{Grunbaum1963} showed that fullerene isomer with $n$ atoms exists for $n=20$ and for all even  $n>22$.
To analyze the performance of the Clar number as a stability predictor of the fullerene isomers, we need good
upper bounds on the Clar number of fullerenes. Fortunately,
Zhang and Ye  \cite{Zhang2007} established an upper bound of the Clar number of
fullerenes. An alternative proof was given by Hartuny
\cite{Hartung2013}.

\begin{thm} {\em \cite{Zhang2007}}\label{Clar}
Let $F$ be a fullerene with $n$ vertices. Then
$c(F)\leq\lfloor n/6\rfloor-2$.
\end{thm}

There are seven experimentally produced  fullerenes attaining the bound
in Theorem \ref{Clar}, namely, C$_{60}$:1 $(I_h)$ \cite{Kroto1985}, C$_{70}$:1 $(D_{5h})$ \cite{Taylor1990}, C$_{76}$:1 $(D_2)$, \cite{Ettl1991,Taylor1993}
C$_{78}$:1 $(D_3)$ \cite{Taylor1993,Diederich1991,Kikuchi1992}, C$_{82}$:3 $(C_2)$ \cite{Kikuchi1992},
 C$_{84}$:22 $(D_2)$ \cite{Taylor1993,Manolopoulos1992} and C$_{84}$:23 $(D_{2d})$ \cite{Taylor1993,Manolopoulos1992}, where
 C$_n$:$m$  occurs at position $m$
in a list of lexicographically ordered spirals that describe isolated-pentagon isomers with $n$ atoms \cite{Fowler1995}, and the point group of the
isomer is presented inside parenthesis.
  Ye and Zhang \cite{Ye2009} gave a graph-theoretical characterization of
  fullerenes with at least 60 vertices attaining the maximum Clar number $n/6-2$,
  and constructed all 18 fullerenes
 attaining the maximum value 8 among all 1812 fullerene
isomers of C$_{60}$. Later, Zhang et al. \cite{Zhang2010b} proposed a
combination of the Clar number and Kekul\'{e} count to predict the
stability of fullerenes, which distinguishes uniquely the
buckminsterfullerene C$_{60}$ from its all 1812 fullerene isomers. Recently, Hartung \cite{Hartung2013} gave another graph-theoretical characterization of
  fullerenes,  whose  Clar numbers are $n/6-2$, by establishing a connection between fullerenes and (4,6)-fullerenes, where a \textsf{(4,6)-fullerene}
  is a trivalent plane graph consisting solely of quadrilaterals and hexagons and is the molecular graph of some possible boron-nitrogen fullerene \cite{Fowler1996}.

In this paper, we will show that there are no fullerenes with $n\equiv 2\pmod 6$ vertices attaining this bound. Thus Theorem
\ref{Clar} is refined as the following theorem.

\begin{thm}\label{new}
Let $F$ be a fullerene with $n$ vertices. Then
\begin{equation*}
c(F)\leq\left\{\begin{array}{ll}
\lfloor n/6\rfloor-3,\qquad n\equiv 2\pmod 6;\\[2ex]
\lfloor n/6\rfloor-2,\qquad otherwise.
\end{array}
\right.
\end{equation*}
\end{thm}

We say a fullerene \textsf{extremal} if the Clar number of the
fullerene attains the bound in Theorem \ref{new}. In addition to the seven
experimentally produced  extremal fullerenes mentioned before,
 there are two experimentally produced extremal fullerenes
C$_{80}$:1 $(D_{5d})$ \cite{Hennrich1996,Wang2000}, C$_{80}$:2 $(D_2)$ \cite{Hennrich1996} (see Figure \ref{fig.11}). Moreover, the
minimum fullerene C$_{20}$ is also an extremal fullerene.

\begin{figure}[h]
\begin{center}
\includegraphics{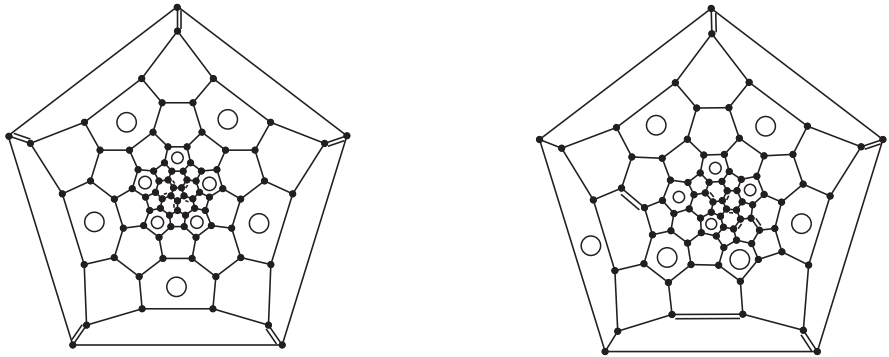}
\end{center}
\caption{Two experimentally produced extremal fullerenes (a) C$_{80}$:1 $(D_{5d})$, (b) C$_{80}$:2 $(D_2)$. (These two graphs are generated by a software package \cite{Schwerdtfeger2013} for constructing and analyzing
structures of fullerenes before further processing.)}
\protect\label{fig.11}
\end{figure}

Furthermore, we give a
graph-theoretical characterization of fullerenes, whose order $n$ is congruent
to 2 (respectively, 4) modulo 6, attaining the maximum Clar number
$\lfloor n/6\rfloor-3$ (respectively,
$\lfloor n/6\rfloor-2$).

\section{Preliminaries}
\setlength{\unitlength}{1cm}
This section presents some concepts and results to be used later.
For the concepts and notations of graphs not defined, we refer to  \cite{West2001}.

Let $F$ be a fullerene. A \textsf{perfect matching} (or
\textsf{Kekul\'{e} structure}) $M$ of $F$ is a set of edges such
that each vertex is incident with exactly one edge in $M$. The faces
with exactly half of their bounding edges in a perfect matching
$M$ of $F$ are called \textsf{alternating faces} with respect to $M$. A \textsf{resonant pattern} of $F$ is a set of
independent alternating faces with respect to some perfect matching. The \textsf{Clar number}
$c(F)$ of $F$ is the maximum size of all resonant patterns of $F$.
  A \textsf{Clar set} is a set of
independent alternating faces of size $c(F)$. If $\mathcal{H}$ is a resonant pattern of $F$ and $M_0$ is a perfect
matching of $F-\mathcal{H}$, then we say $(\mathcal{H},M_0)$ is
a \textsf{Clar cover} \cite{Zhang1996} of $F$. We say a Clar cover $(\mathcal{H},M_0)$ is a \textsf{Clar structure} if $\mathcal{H}$ is a Clar set of $F$.

\begin{figure}[h]
    \begin{center}
\includegraphics[width=12cm]{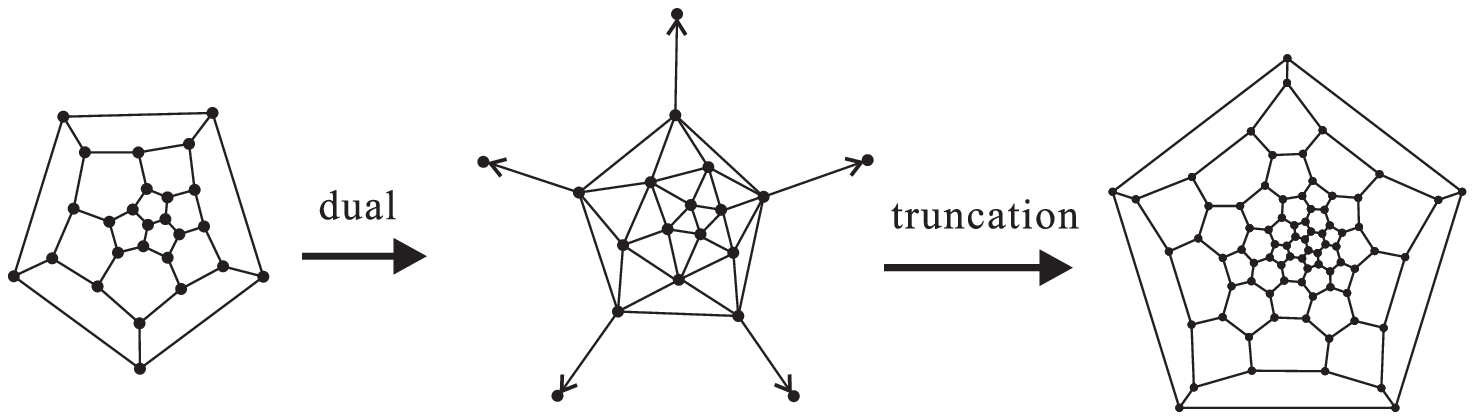}
\end{center}
   \caption{Illustration for the generation procedure of a fullerene with 78 vertices from another fullerene with 26 vertices
by Leapfrog transformation. In order to make the procedure geometrically intuitive, one vertex of the graph in the middle has to be chosen at infinity.}
\protect\label{fig.2}
\end{figure}

\textsf{Leapfrog transformation}  for a 2-connected
plane graph $G$ is usually defined as the truncation of the dual of
$G$ \cite{Godsil2001,Pisanski2000}. The \textsf{leapfrog graph} $\mathcal {L}(G)$ is obtained
from $G$ by performing the  leapfrog
transformation. The \textsf{dual} of a plane graph is built as follows: Place a point in the inner
of each face and join two such points if their corresponding
faces share a common edge \cite{Pisanski2000}.  The \textsf{truncation} of a 2-connected plane graph
$G$ can be obtained by replacing each vertex $v$ of degree $k$ with $k$
new vertices, one for each edge incident to $v$.  Pairs of vertices
corresponding to the edges of $G$ are adjacent, and $k$ new vertices
corresponding to a single vertex of $G$ are joined in the cyclic order given by the embedding to form a face of size $k$ \cite{Godsil2001}. Figure \ref{fig.2} illustrates the generation procedure of a fullerene with 78 vertices from another fullerene with 26 vertices
by leapfrog transformation.
Leapfrog transformation is defined equivalently as the dual of the omnicapping \cite{Fowler1987}.
Leapfrog fullerenes have their own chemical importance. Firstly, they obey the isolated-pentagon rule \cite{Fowler1995}. Secondly,
 they are known to be one of the two constructions that always have properly closed-shell configurations \cite{Fowler1994}. Finally,
 they attain the maximum Fries number $\frac{n}{3}$ and thus are maximally stable in a localised valence bond picture \cite{Fowler1992}.

Let $F$
 be a fullerene and $(\mathcal{H},M)$  a Clar
cover  of  $F$. For a face $f$  of $F$, we say that an edge $e$ in $M$ \textsf{exits}
$f$ if $e$ shares exactly one vertex with $f$.
The following lemma is essentially due to Hartung \cite{Hartung2013}.

\begin{lemma}\cite{Hartung2013}\label{prem}
 Let $F$
 be a fullerene and $(\mathcal{H},M)$  a Clar
cover  of  $F$. Then there are an even number of edges in $M$
(possibly $0$) exiting any hexagon and  an odd number of
edges in $M$ exiting any pentagon.
\end{lemma}

A \textsf{perfect Clar structure} \cite{Fowler1994} (or
\textsf{face-only vertex covering} \cite{Hartung2013}) of a 2-connected
plane graph $G$ is a set of vertex-disjoint faces that include each
vertex of $G$ once. 
The following lemma \cite{Fowler1994}
provides a graph-theoretical characterization of leapfrog graphs on the plane.

\begin{lemma}\cite{Fowler1994}\label{leap}
A $2$-connected plane graph is a leapfrog graph if and only if it is trivalent and
has a perfect Clar structure.
\end{lemma}

Suppose $G$ is a leapfrog graph. Let $\mathcal{P}$ be a perfect Clar structure of $G$. We construct a new plane graph as follows:
For each face of $G$ not belonging to  $\mathcal{P}$, we  allocate a vertex in the inner of it, then connect two vertices with an edge in the resulting graph if their corresponding faces are adjacent in $G$. It is not difficult to see that the leapfrog graph of the resulting graph is $G$. This graph is called  the \textsf{reverse leapfrog} of $G$ determined by $\mathcal{P}$, and
denote it by $\mathcal{L}^{-1}(G, \mathcal{P})$.

\begin{figure}
    \begin{center}
\includegraphics[width=12cm]{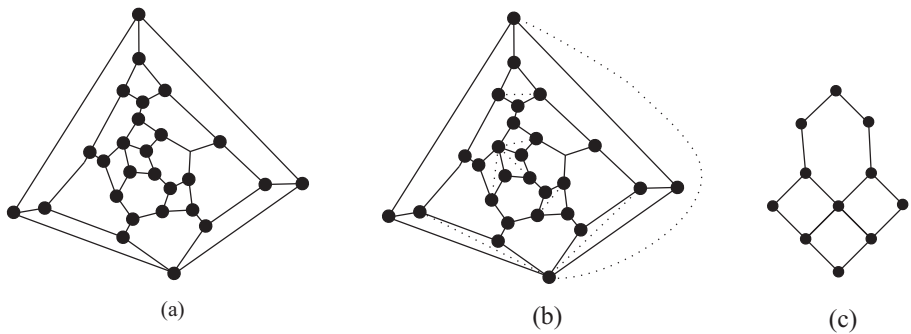}
\end{center}
   \caption{(a) A plane graph; (b) A generalized diagonalization of (a); (c) $T$.}
\protect\label{fig.13}
\end{figure}

Let $G$ be a 2-connected plane graph satisfying : (1) It consists of only  faces of even size,
(2) The maximum degree of $G$ is at most 6 and the minimum degree of $G$ is 3, and
(3) Each vertex of degree 4 is incident with at least 2 quadrilaterals, each vertex of degree 5 is incident with at least 4 quadrilaterals and
 each vertex of degree 6 is incident with exactly 6 quadrilaterals.
  For our purpose we define a \textsf{generalized diagonalization} of  $G$
as a choice of diagonal vertices for each quadrilateral so that each
vertex of degree 4 is chosen twice or thrice, each vertex of degree $5$ is
chosen four or five times, each vertex of degree $6$ is
chosen six times and  any other vertex is chosen at most once. A generalised diagonalization of a plane graph in Figure \ref{fig.13} (a) is illustrated in Figure \ref{fig.13} (b). The generalized diagonalization for such graphs is a natural generalization of
  diagonalization for (4,6)-fullerenes introduced by Hartung \cite{Hartung2013}.  A \textsf{diagonalization} of a (4,6)-fullerene is defined as a choice of diagonal vertices for each quadrilateral so that each vertex is chosen at most once \cite{Hartung2013}.
The following theorem of Hartung \cite{Hartung2013} gives a  graph-theoretical characterization of
 extremal  fullerenes with $n\equiv 0\pmod 6$ vertices.
 \begin{thm}\label{ex6k}\cite{Hartung2013}
 The extremal  fullerenes with $n\equiv 0\pmod 6$ vertices are in one-to-one correspondence with the
diagonalized $(4,6)$-fullerenes with $n/3+4$ vertices.
 \end{thm}

 A \textsf{bipartite graph} is a graph whose vertices can be divided into two disjoint sets such that every edge has an end in each set.
The following lemma is well known.
\begin{lemma}\label{plane}
 A connnected plane graph is bipartite if and only if it has only faces of even size.
\end{lemma}

 A graph $G$ is \textsf{cyclically
$k$-edge-connected} if deleting less than $k$ edges from $G$ can not
separate it into two components such that each of them contains at
least one cycle. The \textsf{cyclical edge-connectivity of $G$}, denote by
$c\lambda(G)$, is the greatest integer $k$ such that $G$ is
cyclically $k$-edge-connected.  Cyclical edge-connectivity plays an important role in handling problems related to fullerenes.
For examples, it is used to study the 2-extendability \cite{Zhang2001} and the lower bound of the forcing number \cite{Zhang2010a} of fullerenes and the hamiltonicity \cite{Maruic2007}
of the leapfrog fullerenes.
Do\v{s}li\'{c} \cite{Doslic2003}, Qi and Zhang  \cite{Qi2008}, and
Kardo\v{s} and \v{S}krekovski \cite{Kardos2008} determined the cyclical edge-connectivity of fullerenes.

\begin{lemma}\cite{Doslic2003,Qi2008,Kardos2008}\label{cyclical}
 For a fullerene $F$, $c\lambda(F)=5.$
\end{lemma}

\section{Proof of Theorem \ref{new}}

Let $F$ be a fullerene and $(\mathcal{H},M)$ a  Clar cover of $F$. Then the
\textsf{expansion} of $F$ is defined as follows: Widen each edge in
$M$ into a quadrilateral. Each vertex covered by $M$ becomes an
edge (see Figure \ref{fig.7}).
Denote this new graph by $\mathscr{E}(\mathcal{H},M)$ and the set of
quadrilaterals
 by $\mathcal{Q}$. The following lemma will be useful.

\begin{figure}
\begin{center}
\includegraphics{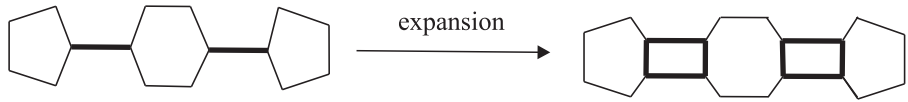}
\end{center}
   \caption{Illustration for expansion.}
\protect\label{fig.7}
\end{figure}

\begin{lemma}\label{extension}
Let $F$ be a fullerene with a Clar cover $(\mathcal{H},M)$. Then
\begin{description}
\item $(1)$  $\mathscr{E}(\mathcal{H},M)$ is a leapfrog graph.
\item $(2)$  $\mathscr{E}(\mathcal{H},M)$ is a bipartite graph.
\item $(3)$  $\mathscr{E}(\mathcal{H},M)$ is face $3$-colorable, that is, its faces can be colored with $3$ colors such that no two adjacent faces receive the same color.
\end{description}
\end{lemma}
\begin{proof}
(1) We can see that $\mathscr{E}(\mathcal{H},M)$ is trivalent and
$\mathcal{H}\cup\mathcal{Q}$ forms a perfect Clar structure of
$\mathscr{E}(\mathcal{H},M)$. Thus by Lemma \ref{leap},
$\mathscr{E}(\mathcal{H},M)$ is a leapfrog graph.

(2) If there are an odd number of edges
in $M$ exiting a face of $f$, then $f$ is a pentagon from Lemma \ref{prem} and changes to an
even face after expansion; If there are an even number of edges in $M$ exiting a
face of $f$, then $f$ is a hexagon from Lemma \ref{prem} and changes to an even face after
expansion. Further, there are additional $|M|$ quadrilaterals after
expansion. Hence all faces of $\mathscr{E}(\mathcal{H},M)$ are
faces of even size. By Lemma \ref{plane}, it follows that $\mathscr{E}(\mathcal{H},M)$ is a bipartite graph.

(3) $\mathscr{E}(\mathcal{H},M)$ is face 3-colorable since
$\mathscr{E}(\mathcal{H},M)$ is trivalent and bipartite. Saaty and
Kainen proved that a trivalent plane graph is face 3-colorable if
and only if it has only faces of even degree \cite{Saaty1977}.
\end{proof}

It is implicit in Fowler's work \cite{Fowler1994} that if a trivalent plane graph is face
3-colorable, then it has unique face 3-coloring up to permutation.
Since $\mathscr{E}(\mathcal{H},M)$ is face
3-colorable, each color class forms a perfect Clar structure. Thus
there are three reverse leapfrog graphs of $\mathscr{E}(\mathcal{H},M)$ determined by the three  perfect Clar
structures of $\mathscr{E}(\mathcal{H},M)$. The one corresponding to the perfect
Clar structure $(\mathcal{H},\mathcal{Q})$ of
$\mathscr{E}(\mathcal{H},M)$, that is, $\mathcal{L}^{-1}((\mathscr{E}(\mathcal{H},M),(\mathcal{H},\mathcal{Q}))$, is called the \textsf{parent} of
$\mathscr{E}(\mathcal{H},M)$, and denoted simply by
$\mathcal{L}^{-1}(\mathscr{E}(\mathcal{H},M))$.
 It has exactly
$|M|$ quadrilaterals and  $|\mathcal{H}|$ hexagons.
The following corollary follows immediately.
\begin{cor}\label{parent}
Let $F$ be a fullerene with a  Clar cover $(\mathcal{H},M)$. Then
\begin{description}
\item $(1)$ $\mathcal{L}^{-1}(\mathscr{E}(\mathcal{H},M))$ is a connected plane bipartite graph.
\item $(2)$ $\mathcal{L}^{-1}(\mathscr{E}(\mathcal{H},M))$ consists solely of quadrilaterals and  hexagons.
\item$(3)$ The degree of each vertex of  $\mathcal{L}^{-1}(\mathscr{E}(\mathcal{H},M))$ is at least $3$ and at most $6$.
 \end{description}
 \end{cor}

\begin{figure}[tphb]
\begin{center}
\includegraphics{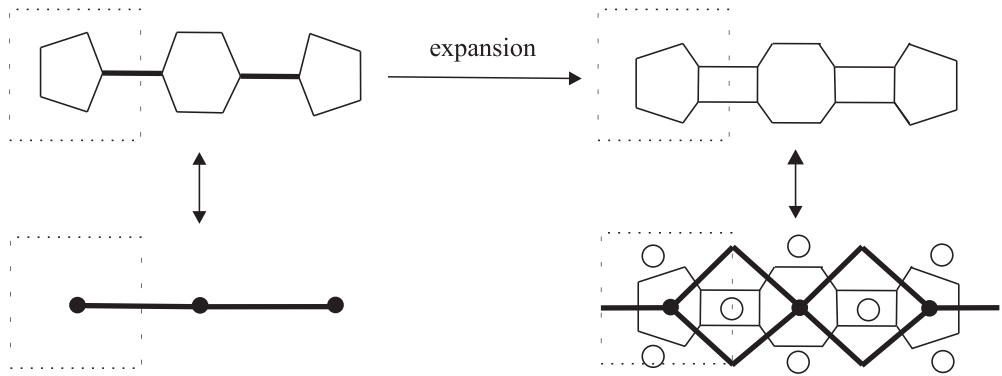}
\end{center}
   \caption{Illustration for the relation between degree of vertices in the $M$-assoiated graph  and the degree of the corresponding vertices in $\mathcal{L}^{-1}(\mathscr{E}(\mathcal{H},M))$.}
\protect\label{fig.14}
\end{figure}

Let $F$ be a fullerene and $(\mathcal{H},M)$ a  Clar cover of $F$.
For every face $f$ of $F$ such that there is at least one edge in $M$ exiting it, we allocate a vertex in the inner of it,  then connect two vertices with an edge if corresponding faces are connected by an edge of $M$ in $F$. The resulting graph is called the \textsf{$M$-associated graph}.

Clearly, each pentagon of $F$ corresponds to a vertex of odd degree in the $M$-associated graph, and each hexagon of $F$ corresponds to a vertex of even degree in the $M$-associated graph.  Since $F$ has exactly 12 pentagons and the other hexagons, the $M$-associated graph  has exactly 12 vertices of odd degree. Furthermore, each edge of $M$ in $F$ corresponds to an edge of the $M$-associated graph. Hence the $M$-associated graph  has $|M|$ edges.
It is clear that each vertex of the $M$-associated graph  corresponds to a vertex in $\mathcal{L}^{-1}(\mathscr{E}(\mathcal{H},M))$. To be more exact,
the
vertices of degree $1$ in the $M$-associated graph  correspond to
vertices of degree $3$ in $\mathcal{L}^{-1}(\mathscr{E}(\mathcal{H},M))$, the
vertices of degrees $2$ and $3$ in the $M$-associated graph
correspond to vertices of degree $4$ in $\mathcal{L}^{-1}(\mathscr{E}(\mathcal{H},M))$, the  vertices of degrees $4$ and $5$ in
the $M$-associated graph correspond to vertices of degree $5$ in $\mathcal{L}^{-1}(\mathscr{E}(\mathcal{H},M))$, and the  vertices of degree $6$
in the associated graph  correspond to  vertices of degree $6$ in $\mathcal{L}^{-1}(\mathscr{E}(\mathcal{H},M))$ (see Figure \ref{fig.14}). We can also see that each vertex in $\mathcal{L}^{-1}(\mathscr{E}(\mathcal{H},M))$ not corresponding to a vertex in the $M$-associated graph has degree 3.

The concept of the $M$-associated graph plays a crucial role in figuring out
the structure of a fullerene with prescribed Clar number. Here is a basic property of the $M$-associated graph.

\begin{lemma}\label{simple}
Let $F$ be a fullerene with a Clar cover $(\mathcal{H},M)$.
Then the $M$-associated graph is simple.
\end{lemma}
\begin{proof}
Suppose to the contrary that the $M$-associated graph contains
multiple edges or loops.
If the $M$-associated graph contains
multiple edges, then there are two  edges
connecting two non-adjacent faces  of $F$.
If the $M$-associated graph contains
 loops, then there is an edge connecting two
non-adjacent vertices of some face of $F$.  In each case, it is easy to
find a cyclic edge-cut of size less than five in $F$, which contradicts
 that $F$ is cyclically 5-edge-connected by Lemma \ref{cyclical}.
\end{proof}

Suppose $F$ is a fullerene with $n\equiv 2\pmod 6$
vertices and $c(F)=\lfloor n/6\rfloor-2$. Let $(\mathcal{H},M)$ be a Clar structure of $F$. Then we have the following result.
\begin{lemma}\label{6k+2rest}
$|M|=7$ and each component of the $M$-associated graph
has $2$ or $4$ vertices of odd degree.
\end{lemma}
\begin{proof}
 It follows directly from $c(F)=\lfloor n/6\rfloor-2$ and
$n\equiv 2\pmod 6$ that $|M|=7$.
 Suppose to the contrary that there is a component $G$ of
the $M$-associated graph  having less than $2$ vertices of odd degree.
Since $G$ has an even number of vertices of odd degree, there are no vertices of odd degree in $G$.
Thus $G$ contains a cycle. By Lemma \ref{simple}, $G$ contains at least 3 edges.
Since  each
component of the $M$-associated graph with $k$ $(k\geq2)$ vertices of odd degree  has at least $k-1\geq \frac{k}{2}$ edges and the $M$-associated graph has exactly 12 vertices of odd degree, the components of the $M$-associated graph other than $G$ have at least $6$ edges.
Hence $M$ has at least 9 edges, which contradicts that $|M|=7$.

On the other hand, suppose to the contrary that there is a component $G$ of
the $M$-associated graph  having more than $4$ vertices of odd degree.
Since the number of vertices of odd degree in $G$ is even, $G$ has at least 6 vertices of odd degree.
Because $|M|=7$, $G$ has at most 8 vertices of odd degree.
If $G$ has exactly 6 vertices of odd degree, then $G$ at least 5 edges. All the other components of the $M$-associated graph  have exactly $6$ vertices of odd degree,
and thus have at least 3 edges. It follows that $M$ has at least $8$ edges,
which contradicts that $M$ has exactly $7$ edges.
If $G$ has exactly $8$ vertices of odd degree, then $G$ has at least $7$ edges, and further
the other components of the $M$-associated graph  have at least 2 edges, which contradicts that
$M$ has exactly $7$ edges.
\end{proof}



By Lemma \ref{6k+2rest},  each component of the $M$-associated graph
has  $2$ or $4$ vertices of odd degree.
If a component $G$ of the $M$-associated graph
has exactly $2$ vertices of odd degree, then the other components of the $M$-associated graph other than $G$
have exactly $10$ vertices of odd degree and at least 5 edges. Since $|M|=7$, $G$ has at most 2 edges. Hence $G$ is $P_2$ or $P_3$.
If a component $G$ of the $M$-associated graph
has exactly $4$ vertices of odd degree, then the other components of the $M$-associated graph other than $G$
have exactly $8$ vertices of odd degree and at least 4 edges. Since $|M|=7$, $G$ has at most 3 edges.
Since $G$ is connected, $G$ has at least 3 edges. Hence $G$ has exactly 3 edges and $G$ is $K_{1,3}$.
So all possible components of the  $M$-associated graph
     are $P_2, P_3,$ and $K_{1,3}$. Suppose
the $M$-associated graph has $n_1$ copies of $P_2$, $n_2$ copies of $P_3$ and $n_3$ copies of $K_{1,3}$ as its components. We have a system of linear indeterminate  equations

\begin{equation}
\left\{\begin{array}{ll} n_1+2n_2+3n_3=7;\\[2ex]
2n_1+2n_2+4n_3=12.
\end{array}
\right.
\end{equation}
 Solving it, we
have the following two solutions: (1) $n_1=5,n_2=1, n_3=0$;
 (2)  $n_1=4,n_2=0, n_3=1$.

In order to prove Theorem \ref{new}, it suffices to prove the following result.
\begin{thm}\label{6k+2nonexist}
Let $F$ be a fullerene with $n\equiv 2\pmod 6$
vertices. Then $c(F)\neq\lfloor\frac{n}{6}\rfloor-2$.
\end{thm}

\begin{proof}
Suppose to the contrary that $c(F)=\lfloor\frac{n}{6}\rfloor-2$. Let $(\mathcal{H},M)$ be a
Clar structure of $F$. Then by the discussion before, the $M$-associated graph is either $P_3\cup 5P_2$ or $K_{1,3}\cup 4P_2$.  In either case, $\mathcal{L}^{-1}(\mathscr{E}(\mathcal{H},M))$  has one vertex of degree $4$  and
 the other vertices of degree $3$. Moreover, $\mathcal{L}^{-1}(\mathscr{E}(\mathcal{H},M))$  consists
seven quadrilaterals and the other  hexagons.


In what follows, we are going to prove such a graph does not exist.
Suppose there is a plane graph $G$ satisfying the above property.  Since  all faces of $G$ are of even size, by Lemma \ref{plane}, $G$ is a bipartite graph.
 Moreover, the connectivity of $G$ guarantees that
 the bipartition is unique. Suppose $G=(A,B)$.
 Without loss of generality, we may assume that the unique 4-degree
vertex is contained in $A$. Then $|E(G)|=3|A|+1=3|B|$, a contradiction.
\end{proof}


\section{Extremal fullerenes with $n\equiv 4\pmod 6$
vertices}
In this section, we give a graph-theoretical characterization of the extremal fullerenes whose
order is congruent to 4  modulo 6. The following lemma is a counterpart of Lemma \ref{6k+2rest}. The proof
is analogous to the corresponding proof of Lemma \ref{6k+2rest} and is omitted
here.

\begin{lemma}\label{6k+4rest}
Let $F$ be an extremal fullerene with $n\equiv 4\pmod 6$
vertices and $(\mathcal{H},M)$ a Clar structure of $F$. Then $|M|=8$ and
 each component of the $M$-associated graph has  $2$, $4$ or $6$ vertices of odd degree.

\end{lemma}
By Lemma \ref{6k+4rest},
we can enumerate all possible components of
the $M$-associated graph.
All
possible components of
the $M$-associated graph  are $P_2, P_3, P_4, K_{1,3}, K_{1,4}$ and $K_{1,5}$. Suppose
the $M$-associated graph has $n_1$ copies of $P_2$, $n_2$ copies of $P_3$, $n_3$ copies of $P_4$,
$n_4$ copies of $K_{1,3}$, $n_5$ copies of $K_{1,4}$ and $n_6$ copies of $K_{1,5}$ as its components. We have a system of linear indeterminate  equations.
\begin{equation}\label{equ3}
\left\{\begin{array}{ll}
n_1+2n_2+3n_3+3n_4+4n_5+5n_6=8;\\[2ex]
2n_1+2n_2+2n_3+4n_4+4n_5+6n_6=12.
\end{array}
\right.
\end{equation}
Solving it, we have the
following 6 solutions: (1) $n_6=1,n_1=3, n_i=0, i\neq1,6$; (2)
$n_5=1,n_1=4, n_i=0, i\neq1,5$; (3) $n_4=2,n_1=2, n_i=0, i\neq1,4$;
(4) $n_4=1,n_2=1, n_1=3, n_i=0, i\neq1,2,4$; (5) $n_3=1,n_1=5, n_i=0, i\neq1,3$;
(6) $n_2=2,n_1=4, n_i=0, i\neq1,2$. Each solution of
the system of linear indeterminate  equations (\ref{equ3}) corresponds to a
possible $M$-associated graph. Now it is time to determine
the structure of $\mathcal{L}^{-1}(\mathscr{E}(\mathcal{H},M))$.



\begin{thm}\label{6k+4}
Let $F$ be an extremal fullerene with $n\equiv 4\pmod 6$
vertices and $(\mathcal{H},M)$ a Clar structure of $F$. Then  $\mathcal{L}^{-1}(\mathscr{E}(\mathcal{H},M))$
is a plane graph with
$(n+14)/3$ vertices satisfying (i) It consists exactly $8$
quadrilaterals and the other hexagons,
 (ii) It has exactly two vertices of degree $4$ and the other
vertices of degree $3$,  (iii) Each vertex of degree $4$ is incident with at least
$2$ quadrilaterals, and (iv) Each partite set contains one
vertex of degree $4$.
\end{thm}

\begin{proof}
Consider all the 6  possible $M$-associated graphs corresponding to all the 6 solutions of the
system of linear indeterminate  equations (\ref{equ3}).
 For each possible $M$-associated graph, $\mathcal{L}^{-1}(\mathscr{E}(\mathcal{H},M))$
consists exactly 8 quadrilaterals and the other  hexagons.

 For each possible $M$-associated graph
 corresponding to Solutions (1) and (2) of the
system of linear indeterminate  equations (\ref{equ3}),
 $\mathcal{L}^{-1}(\mathscr{E}(\mathcal{H},M))$ has
 exactly one vertex of degree $5$ and the other vertices of degree $3$.
 We will show that such a graph does not exist. If not, suppose $G$ is such a plane graph.  Since
all faces of $G$ are of even size, by Lemma \ref{plane}, $G$ is a bipartite graph. Because $G$ is
connected, the bipartition of $G$ is unique. Suppose $G=(A,B)$.
Without loss of generality, we may assume that the unique
vertex of degree $5$ is contained in $A$. Then $|E(G)|=3|A|+2=3|B|$,  a contradiction.

 For each possible $M$-associated graph
 corresponding to the other solutions of the
system of linear indeterminate  equations (\ref{equ3}),
 $\mathcal{L}^{-1}(\mathscr{E}(\mathcal{H},M))$ has exactly two
  vertices of degree $4$ and the other vertices of degree $3$. Amongst these, for each possible $M$-associated graph
 corresponding to Solution (5), two vertices of degree
$4$ in $\mathcal{L}^{-1}(\mathscr{E}(\mathcal{H},M))$ belong to the same partite set, whereas for each possible $M$-associated graph corresponding to
 the other solutions, two vertices of degree
$4$ in $\mathcal{L}^{-1}(\mathscr{E}(\mathcal{H},M))$ belong to different partite sets.
We claim that  two  vertices of degree 4 in $\mathcal{L}^{-1}(\mathscr{E}(\mathcal{H},M))$
belong to different partite sets.
Suppose to the contrary that two 4-degree
vertices of $G$ belong to the same partite set, say $A$. Then
$|E(G)|=3|A|+2=3|B|$, a contradiction.
  Moreover, we can see that each vertex of degree $4$ is incident with at least
$2$ quadrilaterals.
\end{proof}

By the proof of Theorem \ref{6k+4}, for an extremal fullerene $F$ with $n\equiv 4\pmod 6$ vertices and a Clar structure $(\mathcal{H},M)$ of  $F$, the
$M$-associated graph is one of the following three graphs: $2K_{1,3}\cup2P_2$, $K_{1,3}\cup P_3\cup3P_2$ and $2P_3\cup4P_2$.

 Let $G$ be a plane graph which satisfies (i),(ii),(iii) and (iv) in
 Theorem \ref{6k+4}.
 We want to determine the conditions under which
$G$ can be
 transformed into an extremal fullerene.
First we know that $\mathcal {L}(G)$ is a trivalent plane
graph  satisfying $(i')$ It consists of exactly $2$ octagons, $8$
quadrilaterals and the other hexagons,
  $(ii')$ All quadrilaterals of it lie in the same perfect Clar structure, and
  $(iii')$ Each octagon is adjacent to at least two quadrilaterals.
  For each quadrilateral, we select a pair of opposite edges and contract
  it into an edge. Clearly, the resulting graph is trivalent. Since
  quadrilaterals in $G$ correspond to
quadrilaterals in $\mathcal {L}(G)$, selecting a pair of opposite edges
from a quadrilateral in $\mathcal {L}(G)$ is  equivalent to selecting a
pair of faces in $\mathcal {L}(G)$ which are connected by this pair of
opposite edges. Thus it is further equivalent to choose distinct
diagonal vertices from corresponding quadrilateral in $G$. Since the
size of each face connected by a pair of opposite edges in $\mathcal
{L}(G)$ is decreased by one after contraction, in order to make the
size of each face of the resulting graph be five or six, for each
octagon, we should select two or three pair of opposite edges
exiting it. We can contract $\mathcal {L}(G)$ into an extremal
fullerene exactly when a generalized  diagonalization of $G$ is
possible. Together with Theorem \ref{6k+4}, we have the following
graph-theoretical characterization for extremal fullerenes with $6k+4$ vertices.

\begin{figure}[h]
    \begin{center}
\includegraphics{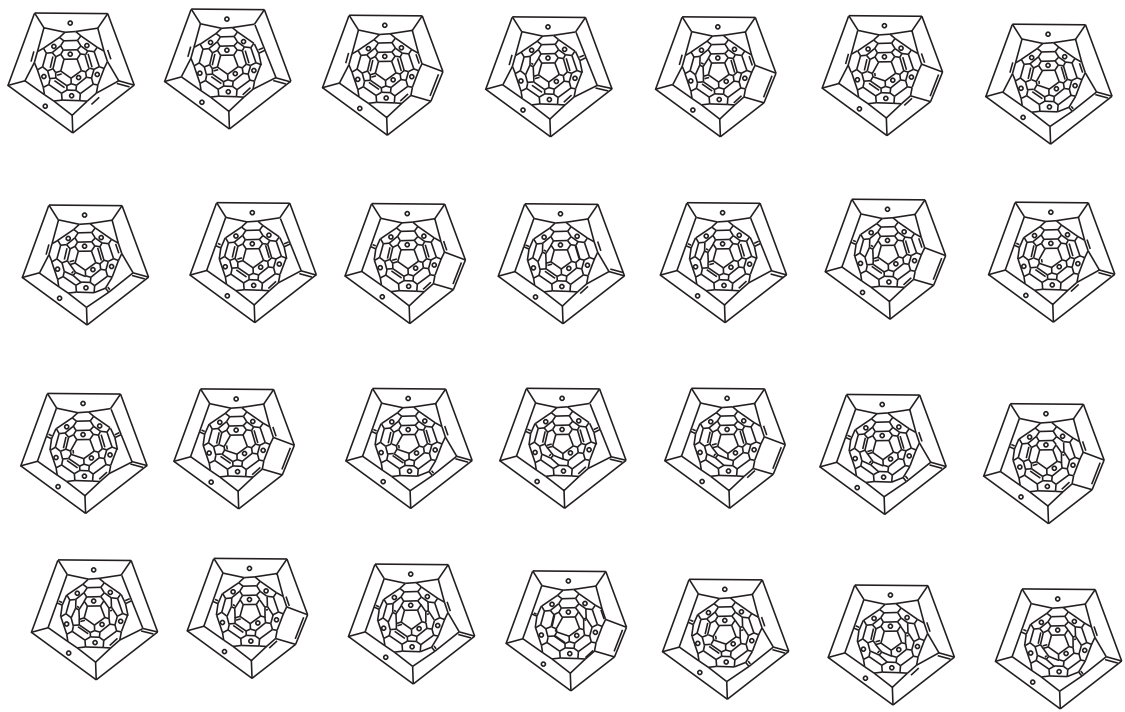}
\end{center}
   \caption{Extremal fullerenes generated from different generalized
   diagonalizations of the plane graph in Figure \ref{fig.13} (a).}
\protect\label{fig.10}
\end{figure}

\begin{thm}\label{6k+4ch}
The extremal fullerenes with $n\equiv 4\pmod 6$ vertices are in
one-to-one correspondence with the generalized diagonalized plane
graphs with $(n+14)/3$ vertices satisfying $(i)$ It consists
exactly $8$ quadrilaterals and the other hexagons,
 (ii)  It has exactly two vertices of degree $5$ and the other
vertices of degree $3$,  (iii) Each vertex of degree $4$ is incident with at least
$2$ quadrilaterals, and (iv) Each partite set of $G$ contains one
vertex of degree $4$.
\end{thm}

As an example, consider a graph illustrated in Figure \ref{fig.13}
(a). This graph satisfies Conditions (i), (ii), (iii) and (iv)
in Theorem \ref{6k+4}. Further, it  has multiple disjoint two
quadrilaterals and two copies of $T$ (see Figure \ref{fig.13} (c))
as  subgraphs. Since a quadrilateral has two ways to diagonalize and a copy of
$T$ has four ways to diagonalize, we have totally 64 different generalized
diagonalization of $G$. Each corresponds to an extremal fullerene.
Eliminating the isomorphic fullerenes from these 64 extremal
fullerenes, we have 28 multiple non-isomorphic extremal fullerenes
(see Figure \ref{fig.10}), including the experimentally produced
C$_{70}$:1$(D_{5h})$.

\section{Extremal fullerenes with $n\equiv 2\pmod 6$
vertices}

In this section, we study the extremal fullerenes whose
order is congruent to 2  modulo 6. By a similar argument as in Lemma \ref{6k+2rest}, we have the following result.

\begin{lemma}\label{2pen}
Let $F$ be an extremal fullerene with $n\equiv 2\pmod 6$
vertices and $(\mathcal{H},M)$ a Clar structure of $F$. Then $|M|=10$ and
each component of the $M$-associated graph has at most $10$ vertices of odd degree.
\end{lemma}

Note that there are a number of candidates for possible components of the $M$-associated graph satisfying the constraint in Lemma \ref{2pen}.
The following lemma reduces the number of candidates for possible components of the $M$-associated graph considerable.

\begin{lemma}\label{2p5}
Let $F$ be an extremal fullerene with $n\equiv 2\pmod 6$
vertices and $(\mathcal{H},M)$ a Clar structure of $F$.  Then
 the $M$-associated graph does not contain $P_5$ as subgraph.
\end{lemma}
\begin{proof}
It suffices to show that there is no component of the $M$-associated graph containing $P_5$ as subgraph.
Suppose to the contrary that there is a component $G$ of the $M$-associated graph
containing $P_5$ as subgraph. Suppose $G$ contains $2s$
vertices of odd degree. Then by Lemma \ref{2pen}, we have $1\leq s\leq 5.$

\noindent{\bf Claim 1.} $G$ has at least
$s+3$ edges and at most $s+4$ edges.

For
 the case $s=1$,  $G$ is a path.
Each path containing $P_5$ as its subgraph  has at least  $s+3$ edges. Furthermore, every graph which
 has exactly $2(s+1)$ odd-degree vertices and  contains $P_5$ as its subgraph can be obtained
 from a graph which has exactly $2s$  vertices of odd degree
 and  contains $P_5$ as its subgraph
 in terms of one of the following two operations:
(1) Add an edge between two  vertices of even degree; (2) Add a new
vertex and add an edge between this new vertex and a
vertex of odd degree. For each operation, at least one edge is added. Thus the
associated graph of $G$ has at least $s+3$ edges.

On the other hand, suppose  $G$ has more than
$s+4$ edges. Then the components of the $M$-associated graph except $G$
have less than $6-s$ edges since $|M|=10$. But the components of the $M$-associated graph
except $G$  have exactly $12-2s$ vertices of odd degree, and thus have
at least $6-s$ edges, a contradiction. This proves Claim 1.



By Claim 1£¬ there are two cases to be considered:

Case 1.  $G$ has $s+3$ edges.

\begin{figure}[tphb]
    \begin{center}
\includegraphics{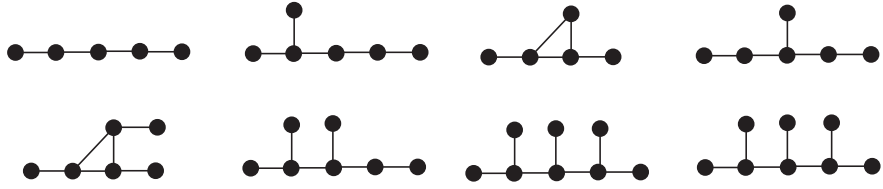}
\end{center}
   \caption{All possible $G$ with $s+3$ edges.}
\protect\label{fig.6}
\end{figure}

\begin{figure}
    \begin{center}
\includegraphics{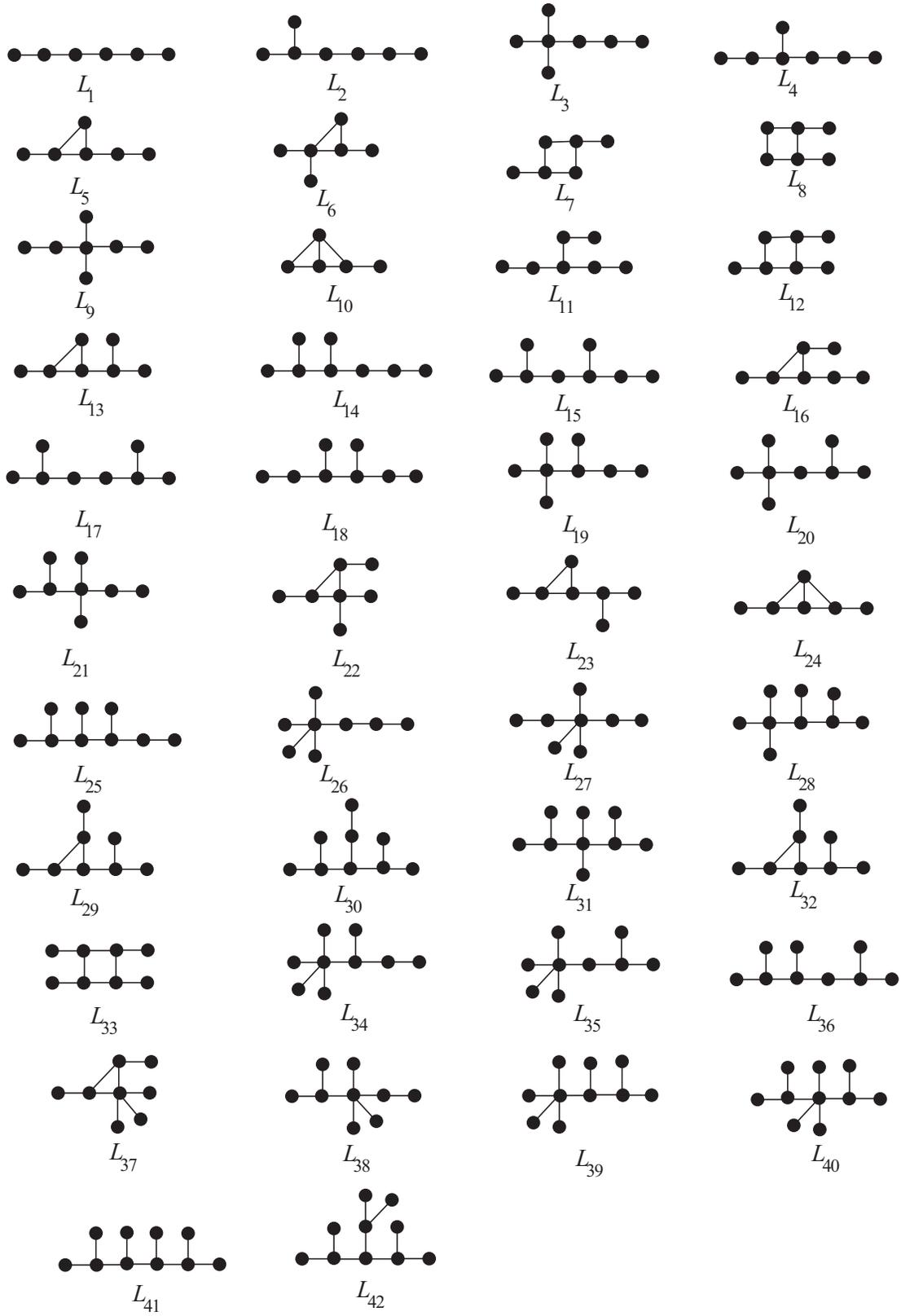}
\end{center}
   \caption{All possible $G$ with $s+4$ edges.}
\protect\label{fig.8}
\end{figure}

All possible  $G$'s are illustrated in Figure
\ref{fig.6}. In each case,
$\mathcal{L}^{-1}(\mathscr{E}(\mathcal{H},M))$ has three vertices of degree $4$
corresponding to three non-1-degree vertices of  $G$.
These three vertices of degree $4$ are contained in the same partite set
of $\mathcal{L}^{-1}(\mathscr{E}(\mathcal{H},M))$. The components of the $M$-associated graph except $G$ have $7-s$ edges and exactly $12-2s$
vertices of odd degree. By a similar discussion  as in
Section 3, there is exactly a new
vertex of degree 4 in $\mathcal{L}^{-1}(\mathscr{E}(\mathcal{H},M))$
except the three vertices of degree $4$ we already have. Since the $\mathcal{L}^{-1}(\mathscr{E}(\mathcal{H},M))$
 is bipartite and connected, we may assume
 $\mathcal{L}^{-1}(\mathscr{E}(\mathcal{H},M))=(A,B)$. Without loss
of generality, we may assume that this new vertex of degree 4 is
contained in $A$. Hence $|\nabla(A)|\equiv1 \pmod 3$, but
$|\nabla(B)|\equiv0 \pmod 3$, where $\nabla(S)$ denotes the set of edges going from $S$ to $V(G)-S$, a contradiction.

Case 2.  $G$ has $s+4$ edges.

All possible $G$'s are illustrated in Figure
\ref{fig.8}. The components of the $M$-associated graph except $G$ have
$6-s$ edges. But the components of the $M$-associated graph except $G$
have exactly $12-2s$ vertices of odd degree. Thus the components of the $M$-associated graph except $G$
are all $P_2$'s.
Hence
there are no vertices of degrees $4$, $5$ and $6$ in $\mathcal{L}^{-1}(\mathscr{E}(\mathcal{H},M))$
except those correspond to $G$. If
$G$ is one of the graphs in $L_1$, $L_2$, $L_4$, $L_5$, $L_7$, $L_8, L_{10},
L_{11}, L_{12}, L_{13}, L_{14}, L_{15}$, $L_{16}$, $L_{17}$, $L_{18}$,
$L_{23}$, $L_{24},$ $L_{25}, L_{29}, L_{30}$, $L_{32}$, $L_{33}$, $L_{36}$,
$L_{41}$ and $L_{42}$, then $\mathcal{L}^{-1}(\mathscr{E}(\mathcal{H},M))$ has exactly four vertices of degree $4$.
These four vertices of degree $4$ are contained in the same partite set
of $\mathcal{L}^{-1}(\mathscr{E}(\mathcal{H},M))$.  We may  assume
$\mathcal{L}^{-1}(\mathscr{E}(\mathcal{H},M))=(A,B)$. Without loss
of generality, we may also assume that these four vertices of degree $4$ are
contained in $A$. Then $|\nabla(A)|\equiv1 \pmod 3$, but
$|\nabla(B)|\equiv0 \pmod 3$, a contradiction. Otherwise, $\mathcal{L}^{-1}(\mathscr{E}(\mathcal{H},M))$ has exactly two
vertices of degree $4$ and one vertex of degree $5$. These two 4-degree
vertices and one vertex of degree $5$ are contained in the same partite
set of $\mathcal{L}^{-1}(\mathscr{E}(\mathcal{H},M))$, say $A$. Then
$|\nabla(A)|\equiv1 \pmod 3$, but $|\nabla(B)|\equiv0 \pmod 3$, a
contradiction.
\end{proof}

By Lemma \ref{2p5}, all possible components of the $M$-associated graph are enumerated in Figure $\ref{fig.5}$.
Suppose the $M$-associated graph has $n_i$ copies of $N_i$ as its components, and $N_i$
has $l_i$ edges  and $k_i$  vertices of odd degree. We have a
system of linear indeterminate  equations:
\begin{equation}\label{equ4}
\left\{\begin{array}{ll}
\sum_{i=1}^{22}n_il_i=10 ;\\[2ex]
\sum_{i=1}^{22}n_ik_i=12.
\end{array}
\right.
\end{equation}
Solving it, we have the
following 45 solutions:
(1) $n_{22}=1,n_3=1, n_i=0, i\neq3,22$;
(2) $n_{21}=1,n_3=2, n_i=0, i\neq3,21$;
(3) $n_{20}=1,n_3=2, n_i=0, i\neq3,20$;
(4) $n_{19}=1,n_9=1, n_i=0, i\neq9,19$;
(5) $n_{19}=1,n_4=1, n_3=1, n_i=0, i\neq3,4,19$;
(6) $n_{18}=1,n_3=3, n_i=0, i\neq3,18$;
(7) $n_{17}=1,n_9=1, n_3=1, n_i=0, i\neq3,9,17$;
(8) $n_{17}=1,n_4=1, n_3=2, n_i=0, i\neq3,4,17$;
(9) $n_{16}=1,n_9=1, n_3=1, n_i=0, i\neq3,9,16$;
(10) $n_{16}=1,n_4=1, n_3=2, n_i=0, i\neq3,4,16$;
(11) $n_{15}=2, n_i=0, i\neq15$;
(12) $n_{15}=1,n_{11}=1, n_3=1, n_i=0, i\neq3,11,15$;
(13) $n_{15}=1,n_{10}=1, n_3=1, n_i=0, i\neq3,10,15$;
(14) $n_{15}=1,n_9=1, n_4=1, n_i=0, i\neq4,9,15$;
(15) $n_{15}=1,n_5=1, n_3=2, n_i=0, i\neq3,5,15$;
(16) $n_{15}=1,n_4=2, n_3=1, n_i=0, i\neq3,4,15$;
(17) $n_{14}=1,n_3=4, n_i=0, i\neq3,14$;
(18) $n_{13}=1,n_3=4, n_i=0, i\neq3,14$;
(19) $n_{12}=1,n_9=1, n_3=2, n_i=0, i\neq3,9,12$;
(20) $n_{12}=1,n_4=1, n_3=3, n_i=0, i\neq3,4,12$;
(21) $n_{11}=2, n_3=2, n_i=0, i\neq3,11$;
(22) $n_{11}=1,n_{10}=1, n_3=2, n_i=0, i\neq3,10,11$;
(23) $n_{11}=1,n_9=2, n_i=0, i\neq9,11$;
(24) $n_{11}=1,n_9=1, n_4=1, n_3=1, n_i=0, i\neq3,4,9,11$;
(25) $n_{11}=1, n_5=1, n_3=3, n_i=0, i\neq3,5,11$;
(26) $n_{11}=1, n_4=2, n_3=2, n_i=0, i\neq3,4,11$;
(27) $n_{10}=2, n_3=2, n_i=0, i\neq3,10$;
(28) $n_{10}=1, n_9=2, n_i=0, i\neq9,10$;
(29) $n_{10}=1, n_9=1, n_4=1, n_3=1, n_i=0, i\neq3,4,9,10$;
(30) $n_{10}=1, n_5=1, n_3=3, n_i=0, i\neq3,5,10$;
(31) $n_{10}=1, n_4=2, n_3=2, n_i=0, i\neq3,4,10$;
(32) $n_9=2, n_5=1, n_3=1, n_i=0, i\neq3,5,9$;
(33) $n_9=2, n_4=2, n_i=0, i\neq4,9$;
(34) $n_9=1, n_6=1, n_3=3, n_i=0, i\neq3,6,9$;
(35) $n_9=1, n_5=1, n_4=1, n_3=2, n_i=0, i\neq3,4,5,9$;
(36) $n_9=1, n_4=3, n_3=1, n_i=0, i\neq3,4,9$;
(37) $n_8=1, n_3=5, n_i=0, i\neq3,8$;
(38) $n_7=1, n_3=5, n_i=0, i\neq3,7$;
(39) $n_6=1, n_4=1, n_3=4, n_i=0, i\neq3,4,6$;
(40) $n_5=2, n_3=4, n_i=0, i\neq3,5$;
(41) $n_5=1, n_4=2, n_3=1, n_i=0, i\neq3,4,5$;
(42) $n_4=4, n_3=2, n_i=0, i\neq3,4$
(43) $n_2=1, n_3=6, n_i=0, i\neq1,8$;
(44) $n_1=1, n_9=1,n_3=4, n_i=0, i\neq1,3,9$;
(45) $n_1=1, n_2=1,n_3=5, n_i=0, i\neq1,2,3$.

According to the above solutions, we can determine the structure of $\mathcal{L}^{-1}(\mathscr{E}(\mathcal{H},M))$.

\begin{figure}[tphb]
    \begin{center}
\includegraphics{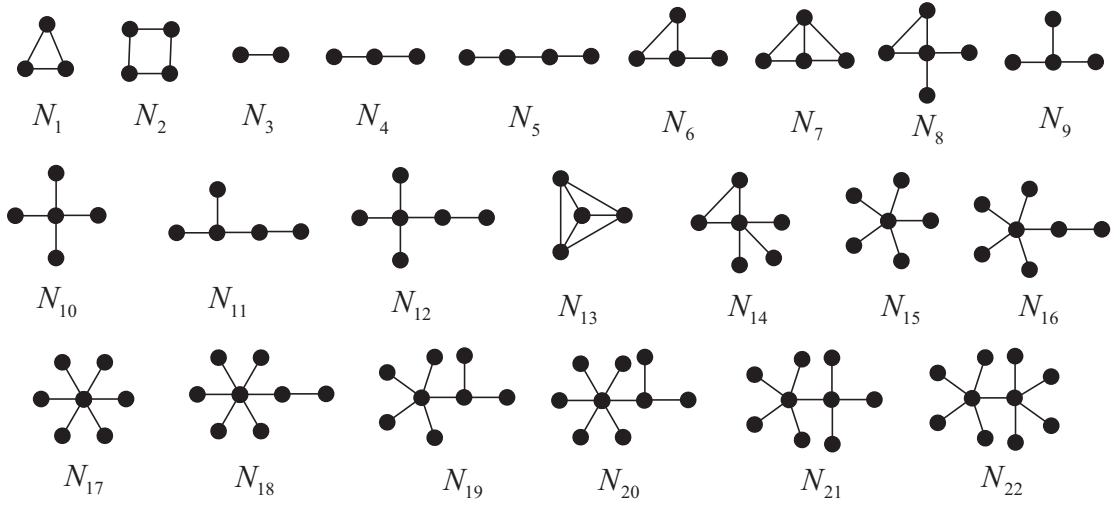}
\end{center}
   \caption{All possible components of
the $M$-associated graph of an extremal fullerene with $n\equiv 2\pmod 6$
vertices.}
\protect\label{fig.5}
\end{figure}

\begin{thm}\label{6k+2}
Let $F$ be an extremal fullerene with $n\equiv 2\pmod 6$
vertices and $(\mathcal{H},M)$ a  Clar structure of $F$.  Then
 $\mathcal{L}^{-1}(\mathscr{E}(\mathcal{H},M))$ is a plane graph with
$(n+16)/{3}$ vertices consisting of exactly $10$ quadrilaterals and
the other hexagons from the following three classes of graphs.
\begin{description}
\item $(1)$ Plane graphs satisfy: (i)  Each has exactly two vertices of degree $5$ and the other
vertices of degree $3$,  (ii) Each vertex of degree $5$ is incident with at least
$4$ quadrilaterals, and (iii) Each partite set  contains one
vertex of degree $5$.
\item $(2)$ Plane graphs satisfy: (i)  Each has exactly two
vertices of degree $4$, one vertex of degree $5$ and the other vertices of degree $3$,  (ii)
The unique vertex of degree $5$ is incident with at least $4$
quadrilaterals, and each vertex of degree $4$ is incident with at
least $2$ quadrilaterals, and (iii) One partite set  contains
two vertices of degree $4$ and the other partite set  contains one
vertex of degree $5$.
\item $(3)$ Plane graphs satisfy: (i)  Each has exactly four
vertices of degree $4$ and the other vertices of degree $3$,  (ii) Each vertex of degree
$4$ is incident with at least $2$ quadrilaterals, and (iii) Each
partite set  contains two vertices of degree $4$.
\end{description}
\end{thm}

\begin{proof}
Consider all the 45  possible $M$-associated graphs corresponding to all the 45 solutions of the
system of linear indeterminate  equations (\ref{equ4}).
 For each possible $M$-associated graph, $\mathcal{L}^{-1}(\mathscr{E}(\mathcal{H},M))$
consists exactly 10 quadrilaterals and the other  hexagons.  For each possible $M$-associated graph
 corresponding to Solutions (3), (6), (7) and (8) of the
system of linear indeterminate  equations (\ref{equ4}),
 $\mathcal{L}^{-1}(\mathscr{E}(\mathcal{H},M))$ has
 exactly one  vertex of degree $6$, one vertex of degree 4 and the other vertices of degree 3.
 For each possible $M$-associated graph
 corresponding to Solutions (1), (2), (11), (13) and (27) of the
system of linear indeterminate  equations (\ref{equ4}),
  $\mathcal{L}^{-1}(\mathscr{E}(\mathcal{H},M))$  has
 exactly two vertices of degree $5$ and other vertices of degree 3.
For each possible $M$-associated graph
 corresponding to Solutions (4), (5), (9), (10), (12), (14), (15), (16), (17), (19), (20), (22), (28), (29), (30), (31) and (37)
  of the
system of linear indeterminate  equations (\ref{equ4}),
$\mathcal{L}^{-1}(\mathscr{E}(\mathcal{H},M))$ has
 exactly one vertex of degree $5$, two vertices of degree $4$ and the other vertices of degree 3.
For each possible $M$-associated graph
 corresponding to the other solutions of the
system of linear indeterminate  equations (\ref{equ4}),
  $\mathcal{L}^{-1}(\mathscr{E}(\mathcal{H},M))$  has exactly
 four
 vertices of degree $4$ and the other vertices of degree 3.

We  proceed to conclude this theorem by proving the following four
claims. Since all graphs discussed in these claims are connected plane graphs having
 only faces of even degree, such graphs are bipartite and have the unique bipartition $(A,B)$.

\noindent{\bf Claim 1.} There is no plane graph satisfying (1) It has exactly one
 vertex of degree 6, one  vertex of degree 4  and the other vertices of degree $3$, and (2) It consists  exactly 10
quadrilaterals and the other  hexagons.

Suppose $G=(A,B)$ is a desired plane graph.
Without loss of generality, we may assume that the unique 4-degree
vertex is contained in $A$. Then $|\nabla(A)|\equiv1 \pmod 3$, but
$|\nabla(B)|\equiv0 \pmod 3$,  a contradiction. This proves Claim 1.

\noindent{\bf Claim 2.} Suppose $G$ is a plane graph satisfying (1) It has
two vertices of degree $5$ and the other vertices of degree 3, and (2) It consists exactly 10 quadrilaterals and the other  hexagons. Then each partite set of $G$
contains one vertex of degree $5$.

  Suppose to the contrary that two 5-degree
vertices are contained in the same partite set, say $A$. Then
$|\nabla(A)|\equiv1 \pmod 3$, but $|\nabla(B)|\equiv0 \pmod 3$, a
contradiction. This proves Claim 2.

\noindent{\bf Claim 3.} Suppose $G$ is a plane graph satisfying (1) It has  exactly two
vertices of degree $4$, one vertex of degree $5$ and the other vertices of degree 3, and
(2) It consists exactly 10 quadrilaterals and the other  hexagons. Then one partite set
of $G$ contains two vertices of degree $4$ and the other partite set of
$G$ contains one vertex of degree $5$.

 Suppose to the contrary that each partite
set of $G$ contains one vertex of degree 4. Without loss of generality,
we may assume that the unique vertex of degree $5$ is contained in $A$.
Then $|\nabla(A)|\equiv0 \pmod 3$, but $|\nabla(B)|\equiv1 \pmod 3$,
a contradiction. This proves Claim 3.

\noindent{\bf Claim 4.} Suppose $G$ is a plane graph satisfying (1) It has  exactly 4
vertices of degree $4$ and the other vertices of degree 3, and (2) It consists exactly 10
quadrilaterals and the other  hexagons. Then each partite set of $G$
contains two vertices of degree 4.

Suppose to the contrary that one partite
set of $G$, say $A$, contains one vertex of degree 4 and the other partite
set $B$ contains
three vertices of degree 4 or one partite set of $G$, say $A$, contains
no vertex of degree 4 and the other partite
set $B$ contains four vertices of degree 4. If $A$
contains one vertex of degree 4 and $B$ contains three 4-degree
vertices, then $|\nabla(A)|\equiv1 \pmod 3$, but $|\nabla(B)|\equiv0
\pmod 3$, a contradiction. If $A$ contains no vertex of degree 4 and
$B$ contains four vertices of degree 4, then $|\nabla(A)|\equiv0 \pmod
3$, but $|\nabla(B)|\equiv1 \pmod 3$, a contradiction. This proves
Claim 4.
\end{proof}

By the proof of Theorem \ref{6k+2}, for an extremal fullerene $F$ with $n\equiv 2\pmod 6$ vertices and a Clar structure $(\mathcal{H},M)$ of  $F$, the
$M$-associated graph is one of the graphs corresponding to the Solutions (11), (12), (13), (14), (15), (16), (21), (22), (23),
(24), (25), (26), (27), (28), (29), (30), (31),  (32), (33), (35), (36), (40), (41) and (42) of the
system of linear indeterminate  equations (\ref{equ4}). Hence only $N_3$, $N_4$, $N_5$, $N_9$, $N_{10}$, $N_{11}$ and $N_{12}$ may serve as components of the $M$-associated graph.

Let $G$ be a plane graph from the three classes of graphs in Theorem
\ref{6k+2}.
 We want to determine the conditions under which $G$ can be
 changed into an extremal fullerene.
It can be easily seen that $\mathcal {L}(G)$ is a trivalent plane graph
satisfying $(i')$ It consists 10 quadrilaterals, and  the others decagons, octagons  and
 hexagons,
  $(ii')$ All quadrilaterals lie in the same perfect Clar structure, and
  $(iii')$ Each decagon is adjacent to at least four quadrilaterals,
  and each octagon is adjacent to at least two quadrilaterals.
  For each quadrilateral, we select a pair of opposite edges and contract
  it into an edge. Clearly, the resulting graph is trivalent.  Analogous to the discussion immediately before Theorem \ref{6k+4ch}, in order to make the resulting graph  be a fullerene,
we should select four or five pair of opposite edges exiting
any octagon and two or three pair of opposite edges exiting  any
octagon.
 We
can contract $\mathcal {L}(G)$ into an extremal fullerene exactly
when a generalized  diagonalization of $G$ is possible. Together
with Theorem \ref{6k+2}, we have the following graph-theoretical characterization for
extremal fullerenes on $6k+2$ vertices.

\begin{thm}\label{6k+2ch}
 The extremal fullerenes with $n\equiv 2\pmod 6$ vertices are in
one-to-one correspondence with the generalized diagonalized plane
graphs with $(n+16)/3$ vertices from the three classes
graphs described in Theorem $\ref{6k+2}$.
\end{thm}

As  two examples, two generalized diagonalized plane graphs in
Figure \ref{fig.12} correspond to two experimentally produced
extremal fullerenes C$_{80}$:1 $(D_5d)$, C$_{80}$:2 $(D_2)$.

\begin{figure}[tphb]
    \begin{center}
\includegraphics{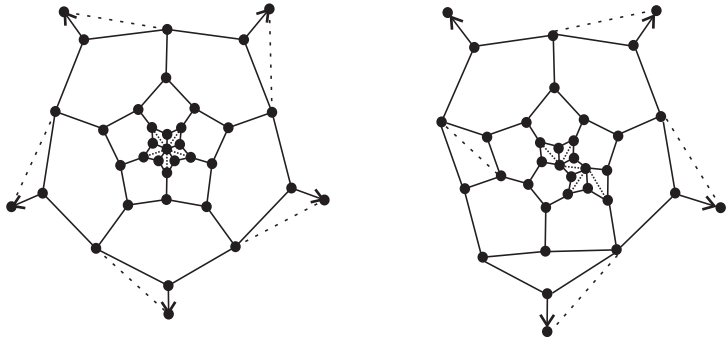}
\end{center}
   \caption{Two generalized diagonalized plane graphs.}
\protect\label{fig.12}
\end{figure}
\section{Conclusions and further directions}

Theorem \ref{6k+4ch} and Theorem \ref{6k+2ch}, together with Theorem \ref{ex6k},
form a complete graph-theoretical characterization
of extremal fullerenes. These results reduce the problem of constructing extremal fullerenes with $n$ vertices
to the problem of constructing some plane graphs with about $n/3$ vertices.

Up to our knowledge, there are only two experimentally produced fullernes, that is,  C$_{78}$:2 $(C_{2v})$ \cite{Diederich1991,Kikuchi1992} and  C$_{78}$:3 $(C_{2v})$ \cite{Kikuchi1992},
are not extremal. Furthermore, there is only one IPR fullerne isomer C$_{78}$:4 $(D_{3h})$, which is not experimentally produced, having larger Clar number than
these two experimentally produced fullernes (see Table \ref{tbl:1}).
However,  Fowler et al. \cite{Fowler1991} predicted that C$_{78}$:4 $(D_{3h})$ is most stable among its all isomers.
Thus Clar number performs well as a stability predictor for IPR fullerenes.
Carr et al. \cite{Carr2014} proved that the Clar number of a fullerene with $n$ vertices
is bounded below by $\lceil(n-380)/61\rceil$. It seems like
this bound can not be attained for any fullerene. In the future work, we will intend to look for an
improved lower bound for the Clar number of fullerenes whose
extremal classes correspond to those ``least stable'' fullerenes.

\begin{table}
  \caption{The Clar numbers of the IPR fullerene isomers with  78 atoms}
  \label{tbl:1}
  \begin{center}
  \begin{tabular}{cc}
    \hline
    Isomer  & Clar number  \\
    \hline
    C$_{78}:1$ $(D_{3})$  & 11   \\
    C$_{78}:2$ $(C_{2v})$ & 10  \\
   C$_{78}:3$ $(C_{2v})$ & 9 \\
   C$_{78}:4$ $(D_{3h})$& 11 \\
   C$_{78}:5 $ $(D_{3h})$& 8 \\
    \hline
  \end{tabular}
  \end{center}
\end{table}

\end{document}